\documentclass{amsart}
\usepackage{amssymb}
\usepackage{amsfonts}

\newtheorem{theorem}{Theorem}
\theoremstyle{plain}
\newtheorem{acknowledgement}{Acknowledgement}

\newtheorem{corollary}{Corollary}

\newtheorem{definition}{Definition}
\newtheorem{example}{Example}

\newtheorem{lemma}{Lemma}

\newtheorem{proposition}{Proposition}
\newtheorem{remark}{Remark}

\numberwithin{equation}{section}
\input{tcilatex}

\begin{document}
\title[Ostrowski type inequalities]{Ostrowski type Inequalities via $h$%
-convex Functions with Applications for Special Means}
\author{Mevl\"{u}t TUN\c{C}}
\address{Kilis 7 Aral\i k University, Faculty of Science and Arts,
Department of Mathematics, Kilis, 79000, Turkey.}
\email{mevluttunc@kilis.edu.tr}

\begin{abstract}
In this paper, we establish some new Ostrowski type inequalities for
absolutely continuous mappings whose first derivatives absolute value are $h$%
-convex (resp. $h$-concave) which are super-multiplicative or
super-additive. Some applications for special means are given.
\end{abstract}

\maketitle

\section{\textbf{Introduction}}

\cite{drr} Let $f:I\subset \left[ 0,\infty \right) \rightarrow 
%TCIMACRO{\U{211d} }%
%BeginExpansion
\mathbb{R}
%EndExpansion
$ be a differentiable mapping on $I^{\circ }$, the interior of the interval $%
I$, such that $f^{\prime }\in L\left[ a,b\right] $, where $a,b\in I$ with $%
a<b$. If $\left\vert f^{\prime }\left( x\right) \right\vert \leq M$, then
the following inequality:%
\begin{equation}
\left\vert f\left( x\right) -\frac{1}{b-a}\int_{a}^{b}f\left( u\right)
du\right\vert \leq \frac{M}{b-a}\left[ \frac{\left( x-a\right) ^{2}+\left(
b-x\right) ^{2}}{2}\right]   \label{101}
\end{equation}%
holds. This result is known in the literature as the Ostrowski inequality.
For recent results and generalizations concerning Ostrowski's inequality,
see \cite{alo,alo2,alo3,cerone,bar,dra} and the references therein.

\begin{definition}
\cite{god}\textit{\ We say that }$f:I\rightarrow 
%TCIMACRO{\U{211d} }%
%BeginExpansion
\mathbb{R}
%EndExpansion
$\textit{\ is Godunova-Levin function or that }$f$\textit{\ belongs to the
class }$Q\left( I\right) $\textit{\ if }$f$\textit{\ is non-negative and for
all }$x,y\in I$\textit{\ and }$t\in \left( 0,1\right) $\textit{\ we have \ \
\ \ \ \ \ \ \ \ \ \ \ }%
\begin{equation}
f\left( tx+\left( 1-t\right) y\right) \leq \frac{f\left( x\right) }{t}+\frac{%
f\left( y\right) }{1-t}.  \label{103}
\end{equation}
\end{definition}

\begin{definition}
\cite{dr1}\textit{\ We say that }$f:I\subseteq 
%TCIMACRO{\U{211d} }%
%BeginExpansion
\mathbb{R}
%EndExpansion
\rightarrow 
%TCIMACRO{\U{211d} }%
%BeginExpansion
\mathbb{R}
%EndExpansion
$\textit{\ is a }$P$-\textit{function or that }$f$\textit{\ belongs to the
class }$P\left( I\right) $\textit{\ if }$f$\textit{\ is nonnegative and for
all }$x,y\in I$\textit{\ and }$t\in \left[ 0,1\right] ,$\textit{\ we have}%
\begin{equation}
f\left( tx+\left( 1-t\right) y\right) \leq f\left( x\right) +f\left(
y\right) .  \label{104}
\end{equation}
\end{definition}

\begin{definition}
\cite{hud}\textit{\ Let }$s\in \left( 0,1\right] .$\textit{\ A function }$f:%
\left[ 0,\infty \right) \rightarrow \left[ 0,\infty \right) $\textit{\ is
said to be }$s$-\textit{convex in the second sense if \ \ \ \ \ \ \ \ \ \ \
\ }%
\begin{equation}
f\left( tx+\left( 1-t\right) y\right) \leq t^{s}f\left( x\right) +\left(
1-t\right) ^{s}f\left( y\right) ,  \label{105}
\end{equation}%
\textit{for all }$x,y\in \left[ 0,\infty \right) $\textit{\ \ and }$t\in %
\left[ 0,1\right] $\textit{. This class of }$s$-\textit{convex functions is
usually denoted by }$K_{s}^{2}$\textit{.}
\end{definition}

\begin{definition}
\cite{var}\textit{\ Let }$h:J\rightarrow 
%TCIMACRO{\U{211d} }%
%BeginExpansion
\mathbb{R}
%EndExpansion
$\textit{\ be a nonnegative function, }$h\not\equiv 0$. \textit{We say that }%
$f:I\subseteq 
%TCIMACRO{\U{211d} }%
%BeginExpansion
\mathbb{R}
%EndExpansion
\rightarrow 
%TCIMACRO{\U{211d} }%
%BeginExpansion
\mathbb{R}
%EndExpansion
$\textit{\ is }$h$-\textit{convex function, or that }$f$\textit{\ belongs to
the class }$SX\left( h,I\right) $\textit{, if }$f$\textit{\ is nonnegative
and for all }$x,y\in I$\textit{\ and }$t\in \left[ 0,1\right] $\textit{\ we
have \ \ \ \ \ \ \ \ \ \ \ \ }%
\begin{equation}
f\left( tx+\left( 1-t\right) y\right) \leq h\left( t\right) f\left( x\right)
+h\left( 1-t\right) f\left( y\right) .  \label{106}
\end{equation}
\end{definition}

If inequality (\ref{106}) is reversed, then $f$ is said to be $h$-concave,
i.e. $f\in SV\left( h,I\right) $. Obviously, if $h\left( t\right) =t$, then
all nonnegative convex functions belong to $SX\left( h,I\right) $\ and all
nonnegative concave functions belong to $SV\left( h,I\right) $; if $h\left(
t\right) =\frac{1}{t}$, then $SX\left( h,I\right) =Q\left( I\right) $; if $%
h\left( t\right) =1$, then $SX\left( h,I\right) \supseteq P\left( I\right) $%
; and if $h\left( t\right) =t^{s}$, where $s\in \left( 0,1\right) $, then $%
SX\left( h,I\right) \supseteq K_{s}^{2}$.

\begin{remark}
\cite{var} Let $h$ be a non-negative function such that%
\begin{equation}
h\left( \alpha \right) \geq \alpha  \label{107}
\end{equation}%
for all $\alpha \in (0,1)$. For example, the function $h_{k}(x)=x^{k}$ where 
$k\leq 1$ and $x>0$ has that property. If $f$ is a non-negative convex
function on $I$ , then for $x,y\in I$ , $\alpha \in (0,1)$ we have 
\begin{equation}
f\left( \alpha x+(1-\alpha )y\right) \leq \alpha f(x)+(1-\alpha )f(y)\leq
h(\alpha )f(x)+h(1-\alpha )f(y).  \label{108}
\end{equation}%
So, $f\in SX(h,I)$. Similarly, if the function $h$ has the property: $%
h(\alpha )\leq \alpha $ for all $\alpha \in (0,1)$, then any non-negative
concave function $f$ belongs to the class $SV(h,I)$.
\end{remark}

\begin{definition}
\cite{var} A function $h:J\rightarrow 
%TCIMACRO{\U{211d} }%
%BeginExpansion
\mathbb{R}
%EndExpansion
$ is said to be a super-multiplicative function if%
\begin{equation}
h\left( xy\right) \geq h\left( x\right) h\left( y\right)  \label{109}
\end{equation}%
for all $x,y\in J,$ when $xy\in J.$
\end{definition}

If inequality (\ref{109}) is reversed, then $h$ is said to be a
sub-multiplicative function. If equality is held in (\ref{109}), then $h$ is
said to be a multiplicative function.

\begin{definition}
\cite{alzer} A function $h:J\rightarrow 
%TCIMACRO{\U{211d} }%
%BeginExpansion
\mathbb{R}
%EndExpansion
$ is said to be a super-additive function if%
\begin{equation}
h\left( x+y\right) \geq h\left( x\right) +h\left( y\right)  \label{110}
\end{equation}%
for all $x,y\in J,$ when $x+y\in J.$
\end{definition}

In \cite{sari}, M.Z. Sar\i kaya, A. Sa\u{g}lam and H. Y\i ld\i r\i m
established the following Hadamard type inequality for $h$-convex functions:

\begin{theorem}
\cite{sari} Let f$\in SX\left( h,I\right) $, $a,b\in I$ and $f\in
L_{1}\left( \left[ a,b\right] \right) $, then%
\begin{equation}
\frac{1}{2h\left( \frac{1}{2}\right) }f\left( \frac{a+b}{2}\right) \leq 
\frac{1}{b-a}\int_{a}^{b}f\left( x\right) dx\leq \left[ f\left( a\right)
+f\left( b\right) \right] \int_{0}^{1}h\left( t\right) dt.  \label{102}
\end{equation}
\end{theorem}

For recent results related $h$-convex functions see \cite%
{bom,bura,oz,sari,var}.

The aim of this study is to establish some Ostrowski type inequalities for
the class of functions whose derivatives in absolute value are $h$-convex
and $h$-concave functions.

\section{\textbf{Ostrowski type inequalities for }$\mathbf{h}$\textbf{%
-convex functions}}

In order to achieve our objective, we need the following lemma \cite{cerone}%
:\qquad

\begin{lemma}
\label{l1}\cite{cerone} Let $f:I\subseteq 
%TCIMACRO{\U{211d} }%
%BeginExpansion
\mathbb{R}
%EndExpansion
\rightarrow 
%TCIMACRO{\U{211d} }%
%BeginExpansion
\mathbb{R}
%EndExpansion
$ be a differentiable mapping on $I^{\circ }$ where $a,b\in I$ with $a<b.$
If $f^{\prime }\in L\left[ a,b\right] $, then the following equality holds;%
\begin{eqnarray*}
&&f\left( x\right) -\frac{1}{b-a}\int_{a}^{b}f\left( u\right) du \\
&=&\frac{\left( x-a\right) ^{2}}{b-a}\int_{0}^{1}tf^{\prime }\left(
tx+\left( 1-t\right) a\right) dt-\frac{\left( b-x\right) ^{2}}{b-a}%
\int_{0}^{1}tf^{\prime }\left( tx+\left( 1-t\right) b\right) dt
\end{eqnarray*}

for each $x\in \left[ a,b\right] .$
\end{lemma}

\begin{theorem}
\label{t2}Let $h:J\subseteq 
%TCIMACRO{\U{211d} }%
%BeginExpansion
\mathbb{R}
%EndExpansion
\rightarrow 
%TCIMACRO{\U{211d} }%
%BeginExpansion
\mathbb{R}
%EndExpansion
$ be a nonnegative and super-multiplicative functions, $f:I\subseteq 
%TCIMACRO{\U{211d} }%
%BeginExpansion
\mathbb{R}
%EndExpansion
\rightarrow 
%TCIMACRO{\U{211d} }%
%BeginExpansion
\mathbb{R}
%EndExpansion
$ be a differentiable mapping on $I^{\circ }$ such that $f^{\prime }\in L%
\left[ a,b\right] $, where $a,b\in I$ with $a<b,$ and $h\left( \alpha
\right) \geq \alpha .$ If $\left\vert f^{\prime }\right\vert $ is $h$-convex
function on $I$ and $\left\vert f^{\prime }\left( x\right) \right\vert \leq
M,$ $x\in \left[ a,b\right] ,$ then we have;%
\begin{equation}
\left\vert f\left( x\right) -\frac{1}{b-a}\int_{a}^{b}f\left( u\right)
du\right\vert \leq \frac{M\left[ \left( x-a\right) ^{2}+\left( b-x\right)
^{2}\right] }{b-a}\int_{0}^{1}\left[ h\left( t^{2}\right) +h\left(
t-t^{2}\right) \right] dt.  \label{21}
\end{equation}%
for each $x\in \left[ a,b\right] .$
\end{theorem}

\begin{proof}
By Lemma \ref{l1} and since $\left\vert f^{\prime }\right\vert $ is $h$%
-convex, then we can write;%
\begin{eqnarray*}
&&\left\vert f\left( x\right) -\frac{1}{b-a}\int_{a}^{b}f\left( u\right)
du\right\vert \\
&\leq &\frac{\left( x-a\right) ^{2}}{b-a}\int_{0}^{1}t\left\vert f^{\prime
}\left( tx+\left( 1-t\right) a\right) \right\vert dt+\frac{\left( b-x\right)
^{2}}{b-a}\int_{0}^{1}t\left\vert f^{\prime }\left( tx+\left( 1-t\right)
b\right) \right\vert dt \\
&\leq &\frac{\left( x-a\right) ^{2}}{b-a}\int_{0}^{1}t\left[ h\left(
t\right) \left\vert f^{\prime }\left( x\right) \right\vert +h\left(
1-t\right) \left\vert f^{\prime }\left( a\right) \right\vert \right] dt \\
&&+\frac{\left( b-x\right) ^{2}}{b-a}\int_{0}^{1}t\left[ h\left( t\right)
\left\vert f^{\prime }\left( x\right) \right\vert +h\left( 1-t\right)
\left\vert f^{\prime }\left( b\right) \right\vert \right] dt \\
&\leq &\frac{M\left( x-a\right) ^{2}}{b-a}\int_{0}^{1}\left[ h^{2}\left(
t\right) +h\left( t\right) h\left( 1-t\right) \right] dt \\
&&+\frac{M\left( b-x\right) ^{2}}{b-a}\int_{0}^{1}\left[ h^{2}\left(
t\right) +h\left( t\right) h\left( 1-t\right) \right] dt \\
&\leq &\frac{M\left[ \left( x-a\right) ^{2}+\left( b-x\right) ^{2}\right] }{%
b-a}\int_{0}^{1}\left[ h\left( t^{2}\right) +h\left( t-t^{2}\right) \right]
dt.
\end{eqnarray*}%
The proof is completed.
\end{proof}

\begin{remark}
In\ (\ref{21}), if we choose $h\left( t\right) =t$, inequality (\ref{21})
reduces to (\ref{101}).
\end{remark}

In the next corollary, we will also make use of the Beta function of Euler
type, which is for $x,y>0$ defined as%
\begin{equation*}
\beta \left( x,y\right) =\int_{0}^{1}t^{x-1}\left( 1-t\right) ^{y-1}dt=\frac{%
\Gamma \left( x\right) \Gamma \left( y\right) }{\Gamma \left( x+y\right) }.
\end{equation*}

\begin{corollary}
In$\ $(\ref{21}), if we choose $h\left( t\right) =t^{s}$, then we have%
\begin{eqnarray*}
\left\vert f\left( x\right) -\frac{1}{b-a}\int_{a}^{b}f\left( u\right)
du\right\vert &\leq &\frac{M\left[ \left( x-a\right) ^{2}+\left( b-x\right)
^{2}\right] }{b-a}\int_{0}^{1}\left[ t^{2s}+\left( t-t^{2}\right) ^{s}\right]
dt \\
&=&\frac{M\left[ \left( x-a\right) ^{2}+\left( b-x\right) ^{2}\right] }{b-a}%
\int_{0}^{1}\left[ t^{2s}+t^{s}\left( 1-t\right) ^{s}\right] dt \\
&=&\frac{M\left[ \left( x-a\right) ^{2}+\left( b-x\right) ^{2}\right] }{b-a}%
\left[ \frac{1}{2s+1}+\frac{\Gamma \left( s+1\right) \Gamma \left(
s+1\right) }{\Gamma \left( 2s+2\right) }\right] \\
&=&\frac{M\left[ \left( x-a\right) ^{2}+\left( b-x\right) ^{2}\right] }{b-a}%
\left[ \frac{\Gamma \left( 2s+1\right) +s^{2}\left( \Gamma \left( s\right)
\right) ^{2}}{\left( 2s+1\right) \Gamma \left( 2s+1\right) }\right]
\end{eqnarray*}
\end{corollary}

One of the important result is given in the following theorem.

\begin{theorem}
\label{t3}Let $h:J\subseteq 
%TCIMACRO{\U{211d} }%
%BeginExpansion
\mathbb{R}
%EndExpansion
\rightarrow 
%TCIMACRO{\U{211d} }%
%BeginExpansion
\mathbb{R}
%EndExpansion
$ be a nonnegative and superadditive functions, $f:I\subseteq 
%TCIMACRO{\U{211d} }%
%BeginExpansion
\mathbb{R}
%EndExpansion
\rightarrow 
%TCIMACRO{\U{211d} }%
%BeginExpansion
\mathbb{R}
%EndExpansion
$ be a differentiable mapping on $I^{\circ }$ such that $f^{\prime }\in L%
\left[ a,b\right] $, where $a,b\in I$ with $a<b.$ If $\left\vert f^{\prime
}\right\vert ^{q}$ is $h$-convex function on $\left[ a,b\right] ,$ $p,q>1,$ $%
\frac{1}{p}+\frac{1}{q}=1$, $h\left( t\right) \geq t$ and $\left\vert
f^{\prime }\left( x\right) \right\vert \leq M$, $x\in \left[ a,b\right] $,
then%
\begin{equation}
\left\vert f\left( x\right) -\frac{1}{b-a}\int_{a}^{b}f\left( u\right)
du\right\vert \leq \frac{Mh^{\frac{1}{q}}\left( 1\right) }{b-a}\left(
\int_{0}^{1}\left( h\left( t^{p}\right) dt\right) \right) ^{\frac{1}{p}%
}\left( \left( x-a\right) ^{2}+\left( b-x\right) ^{2}\right)  \label{22}
\end{equation}%
for each $x\in \left[ a,b\right] .$
\end{theorem}

\begin{proof}
Suppose that $p>1$. From Lemma \ref{l1} and using the H\"{o}lder's
inequality, we can write%
\begin{eqnarray*}
&&\left\vert f\left( x\right) -\frac{1}{b-a}\int_{a}^{b}f\left( u\right)
du\right\vert \\
&\leq &\frac{\left( x-a\right) ^{2}}{b-a}\int_{0}^{1}t\left\vert f^{\prime
}\left( tx+\left( 1-t\right) a\right) \right\vert dt+\frac{\left( b-x\right)
^{2}}{b-a}\int_{0}^{1}t\left\vert f^{\prime }\left( tx+\left( 1-t\right)
b\right) \right\vert dt \\
&\leq &\frac{\left( x-a\right) ^{2}}{b-a}\left( \int_{0}^{1}t^{p}dt\right) ^{%
\frac{1}{p}}\left( \int_{0}^{1}\left\vert f^{\prime }\left( tx+\left(
1-t\right) a\right) \right\vert ^{q}dt\right) ^{\frac{1}{q}} \\
&&+\frac{\left( b-x\right) ^{2}}{b-a}\left( \int_{0}^{1}t^{p}dt\right) ^{%
\frac{1}{p}}\left( \int_{0}^{1}\left\vert f^{\prime }\left( tx+\left(
1-t\right) b\right) \right\vert ^{q}dt\right) ^{\frac{1}{q}}.
\end{eqnarray*}%
Since $\left\vert f^{\prime }\right\vert ^{q}$ is $h$-convex and by using
properties of $h$-convexity in the assumptions,%
\begin{eqnarray*}
\int_{0}^{1}\left\vert f^{\prime }\left( tx+\left( 1-t\right) a\right)
\right\vert ^{q}dt &\leq &\int_{0}^{1}\left[ h\left( t\right) \left\vert
f^{\prime }\left( x\right) \right\vert ^{q}+h\left( 1-t\right) \left\vert
f^{\prime }\left( a\right) \right\vert ^{q}\right] dt \\
&\leq &M^{q}\int_{0}^{1}\left[ h\left( t\right) +h\left( 1-t\right) \right]
dt \\
&\leq &M^{q}\int_{0}^{1}h\left( 1\right) dt=M^{q}h\left( 1\right) .
\end{eqnarray*}%
Similarly, we can show that%
\begin{eqnarray*}
\int_{0}^{1}\left\vert f^{\prime }\left( tx+\left( 1-t\right) b\right)
\right\vert ^{q}dt &\leq &\int_{0}^{1}\left[ h\left( t\right) \left\vert
f^{\prime }\left( x\right) \right\vert ^{q}+h\left( 1-t\right) \left\vert
f^{\prime }\left( b\right) \right\vert ^{q}\right] dt \\
&\leq &M^{q}h\left( 1\right) ,
\end{eqnarray*}%
and%
\begin{equation*}
\int_{0}^{1}t^{p}dt\leq \int_{0}^{1}h\left( t^{p}\right) dt.
\end{equation*}%
Therefore, we obtain%
\begin{eqnarray*}
\left\vert f\left( x\right) -\frac{1}{b-a}\int_{a}^{b}f\left( u\right)
du\right\vert &\leq &Mh^{\frac{1}{q}}\left( 1\right) \frac{\left( x-a\right)
^{2}}{b-a}\left( \int_{0}^{1}h\left( t^{p}\right) dt\right) ^{\frac{1}{p}} \\
&&+Mh^{\frac{1}{q}}\left( 1\right) \frac{\left( b-x\right) ^{2}}{b-a}\left(
\int_{0}^{1}h\left( t^{p}\right) dt\right) ^{\frac{1}{p}} \\
&=&\frac{Mh^{\frac{1}{q}}\left( 1\right) }{b-a}\left( \int_{0}^{1}h\left(
t^{p}\right) dt\right) ^{\frac{1}{p}}\left( \left( x-a\right) ^{2}+\left(
b-x\right) ^{2}\right)
\end{eqnarray*}%
The proof is completed.
\end{proof}

For example, $h\left( t\right) =t^{2}$ is a superadditive function for
nonnegative real numbers because the square of $\left( u+v\right) $ is
always greater than or equal to the square of $u$ plus the square of $v$,
for $u,v\in \left[ 0,\infty \right) $.

\begin{corollary}
\label{c} In (\ref{22}), if we choose $h\left( t\right) =t^{n}$ with $n\in 
%TCIMACRO{\U{2115} }%
%BeginExpansion
\mathbb{N}
%EndExpansion
,$ $n\geq 2,$ then we have%
\begin{equation}
\left\vert f\left( x\right) -\frac{1}{b-a}\int_{a}^{b}f\left( u\right)
du\right\vert \leq \frac{M}{b-a}\left( \frac{1}{np+1}\right) ^{\frac{1}{p}%
}\left( \left( x-a\right) ^{2}+\left( b-x\right) ^{2}\right) .  \label{23}
\end{equation}
\end{corollary}

\begin{remark}
Since $\left( \frac{1}{np+1}\right) ^{\frac{1}{p}}<\frac{1}{2}$, for any $%
4\geq n>p>1,$ $n\in 
%TCIMACRO{\U{2115} }%
%BeginExpansion
\mathbb{N}
%EndExpansion
,$ then we behold that the inequality (\ref{23}) is better than the
inequality (\ref{101}). Better approaches can be obtained even it is
irregular for bigger $n$ and $p$ numbers.
\end{remark}

As we know, $h$-convex functions include all nonnegative convex, $s$-convex
in the second sense, $Q(I)$-convex and $P$-convex function classes. In this
respect, it is normal to obtain weaker results once compared with
inequalities in referenced studies. Because, the inequalities written herein
were considered to be more general than above-mentioned classes and it was
taken into account to be super-multiplicative or super-additive material. In
this case, right side of inequality may be greater.

A new approach for $h$-convex function is given in the following result.

\begin{theorem}
\label{t4}Let $h:J\subseteq 
%TCIMACRO{\U{211d} }%
%BeginExpansion
\mathbb{R}
%EndExpansion
\rightarrow 
%TCIMACRO{\U{211d} }%
%BeginExpansion
\mathbb{R}
%EndExpansion
$ be a nonnegative and supermultiplicative functions, $f:I\subseteq 
%TCIMACRO{\U{211d} }%
%BeginExpansion
\mathbb{R}
%EndExpansion
\rightarrow 
%TCIMACRO{\U{211d} }%
%BeginExpansion
\mathbb{R}
%EndExpansion
$ be a differentiable mapping on $I^{\circ }$ such that $f^{\prime }\in L%
\left[ a,b\right] $, where $a,b\in I$ with $a<b.$ If $\left\vert f^{\prime
}\right\vert ^{q}$ is $h$-convex function on $\left[ a,b\right] ,$ $q\geq 1$%
, $h\left( \alpha \right) \geq \alpha $ and $\left\vert f^{\prime }\left(
x\right) \right\vert \leq M$, $x\in \left[ a,b\right] $, then%
\begin{eqnarray}
&&\left\vert f\left( x\right) -\frac{1}{b-a}\int_{a}^{b}f\left( u\right)
du\right\vert  \label{24} \\
&\leq &\frac{\sqrt[q]{2}M}{2\left( b-a\right) }\left( \left( x-a\right)
^{2}+\left( b-x\right) ^{2}\right) \left( \int_{0}^{1}\left( h\left(
t^{2}\right) +h\left( t-t^{2}\right) \right) dt\right) ^{\frac{1}{q}}  \notag
\end{eqnarray}%
for each $x\in \left[ a,b\right] .$
\end{theorem}

\begin{proof}
Suppose that $q\geq 1$. From Lemma 1 and using the power mean inequality ,
we have%
\begin{eqnarray*}
&&\left\vert f\left( x\right) -\frac{1}{b-a}\int_{a}^{b}f\left( u\right)
du\right\vert \\
&\leq &\frac{\left( x-a\right) ^{2}}{b-a}\int_{0}^{1}t\left\vert f^{\prime
}\left( tx+\left( 1-t\right) a\right) \right\vert dt+\frac{\left( b-x\right)
^{2}}{b-a}\int_{0}^{1}t\left\vert f^{\prime }\left( tx+\left( 1-t\right)
b\right) \right\vert dt \\
&\leq &\frac{\left( x-a\right) ^{2}}{b-a}\left( \int_{0}^{1}tdt\right) ^{1-%
\frac{1}{q}}\left( \int_{0}^{1}t\left\vert f^{\prime }\left( tx+\left(
1-t\right) a\right) \right\vert ^{q}dt\right) ^{\frac{1}{q}} \\
&&+\frac{\left( b-x\right) ^{2}}{b-a}\left( \int_{0}^{1}tdt\right) ^{1-\frac{%
1}{q}}\left( \int_{0}^{1}t\left\vert f^{\prime }\left( tx+\left( 1-t\right)
b\right) \right\vert ^{q}dt\right) ^{\frac{1}{q}}
\end{eqnarray*}%
Since $\left\vert f^{\prime }\right\vert ^{q}$ is $h$-convex, we have%
\begin{eqnarray*}
\int_{0}^{1}t\left\vert f^{\prime }\left( tx+\left( 1-t\right) a\right)
\right\vert ^{q}dt &\leq &\int_{0}^{1}\left[ th\left( t\right) \left\vert
f^{\prime }\left( x\right) \right\vert ^{q}+th\left( 1-t\right) \left\vert
f^{\prime }\left( a\right) \right\vert ^{q}\right] dt \\
&\leq &\left\vert f^{\prime }\left( x\right) \right\vert
^{q}\int_{0}^{1}h\left( t\right) h\left( t\right) dt+\left\vert f^{\prime
}\left( a\right) \right\vert ^{q}\int_{0}^{1}h\left( t\right) h\left(
1-t\right) dt \\
&\leq &M^{q}\left[ \int_{0}^{1}h\left( t^{2}\right) dt+\int_{0}^{1}h\left(
t-t^{2}\right) dt\right] .
\end{eqnarray*}%
Similarly, we can observe that%
\begin{eqnarray*}
\int_{0}^{1}t\left\vert f^{\prime }\left( tx+\left( 1-t\right) b\right)
\right\vert ^{q}dt &\leq &\left\vert f^{\prime }\left( x\right) \right\vert
^{q}\int_{0}^{1}h\left( t\right) h\left( t\right) dt+\left\vert f^{\prime
}\left( b\right) \right\vert ^{q}\int_{0}^{1}h\left( t\right) h\left(
1-t\right) dt \\
&\leq &M^{q}\left\{ \int_{0}^{1}h\left( t^{2}\right) dt+\int_{0}^{1}h\left(
t-t^{2}\right) dt\right\} .
\end{eqnarray*}%
Therefore, we deduce%
\begin{eqnarray*}
&&\left\vert f\left( x\right) -\frac{1}{b-a}\int_{a}^{b}f\left( u\right)
du\right\vert \\
&\leq &\frac{\left( x-a\right) ^{2}}{b-a}\left( \frac{1}{2}\right) ^{1-\frac{%
1}{q}}\left( M^{q}\int_{0}^{1}\left( h\left( t^{2}\right) +h\left(
t-t^{2}\right) \right) dt\right) ^{\frac{1}{q}} \\
&&+\frac{\left( b-x\right) ^{2}}{b-a}\left( \frac{1}{2}\right) ^{1-\frac{1}{q%
}}\left( M^{q}\int_{0}^{1}\left( h\left( t^{2}\right) +h\left(
t-t^{2}\right) \right) dt\right) ^{\frac{1}{q}} \\
&=&M\left( \frac{1}{2}\right) ^{1-\frac{1}{q}}\left( \int_{0}^{1}\left(
h\left( t^{2}\right) +h\left( t-t^{2}\right) \right) dt\right) ^{\frac{1}{q}%
}\left( \frac{\left( x-a\right) ^{2}+\left( b-x\right) ^{2}}{\left(
b-a\right) }\right) \\
&=&\sqrt[q]{2}M\left( \int_{0}^{1}\left( h\left( t^{2}\right) +h\left(
t-t^{2}\right) \right) dt\right) ^{\frac{1}{q}}\left( \frac{\left(
x-a\right) ^{2}+\left( b-x\right) ^{2}}{2\left( b-a\right) }\right)
\end{eqnarray*}%
and the proof is completed.
\end{proof}

\begin{remark}
i) \ In the above inequalities, one can establish several midpoint type
inequalities by letting $x=\frac{a+b}{2}.$

ii) In Theorem \ref{t4}, if we choose

(a) $x=\frac{a+b}{2},$ then we obtain%
\begin{equation*}
\left\vert f\left( \frac{a+b}{2}\right) -\frac{1}{b-a}\int_{a}^{b}f\left(
u\right) du\right\vert \leq \frac{\sqrt[q]{2}M\left( b-a\right) }{4}\left(
\int_{0}^{1}\left( h\left( t^{2}\right) +h\left( t-t^{2}\right) \right)
dt\right) ^{\frac{1}{q}}
\end{equation*}

(b) $x=a,$ then we obtain%
\begin{equation}
\left\vert f\left( a\right) -\frac{1}{b-a}\int_{a}^{b}f\left( u\right)
du\right\vert \leq \frac{\sqrt[q]{2}M\left( b-a\right) }{2}\left(
\int_{0}^{1}\left( h\left( t^{2}\right) +h\left( t-t^{2}\right) \right)
dt\right) ^{\frac{1}{q}}  \notag
\end{equation}

(c) $x=b,$ then we obtain%
\begin{equation}
\left\vert f\left( b\right) -\frac{1}{b-a}\int_{a}^{b}f\left( u\right)
du\right\vert \leq \frac{\sqrt[q]{2}M\left( b-a\right) }{2}\left(
\int_{0}^{1}\left( h\left( t^{2}\right) +h\left( t-t^{2}\right) \right)
dt\right) ^{\frac{1}{q}}  \notag
\end{equation}
\end{remark}

The following result holds for $h$-concave functions.

\begin{theorem}
\label{t5}Let $h:J\subseteq 
%TCIMACRO{\U{211d} }%
%BeginExpansion
\mathbb{R}
%EndExpansion
\rightarrow 
%TCIMACRO{\U{211d} }%
%BeginExpansion
\mathbb{R}
%EndExpansion
$ be a non-negative and superadditive functions, $f:I\subseteq 
%TCIMACRO{\U{211d} }%
%BeginExpansion
\mathbb{R}
%EndExpansion
\rightarrow 
%TCIMACRO{\U{211d} }%
%BeginExpansion
\mathbb{R}
%EndExpansion
$ be a differentiable mapping on $I^{\circ }$ such that $f^{\prime }\in L%
\left[ a,b\right] $ ,where $a,b\in I$ with $a<b.$ If $\left\vert f^{\prime
}\right\vert ^{q}$ is $h$-concave function on $\left[ a,b\right] ,$ $p,q>1,$ 
$\frac{1}{p}+\frac{1}{q}=1$, $h\left( t\right) \geq t$, then%
\begin{eqnarray}
&&\left\vert f\left( x\right) -\frac{1}{b-a}\int_{a}^{b}f\left( u\right)
du\right\vert  \label{25} \\
&\leq &\frac{1}{\sqrt[q]{2}\left( p+1\right) ^{\frac{1}{p}}h^{\frac{1}{q}%
}\left( \frac{1}{2}\right) }\left[ \frac{\left( x-a\right) ^{2}}{b-a}%
\left\vert f^{\prime }\left( \frac{x+a}{2}\right) \right\vert +\frac{\left(
b-x\right) ^{2}}{b-a}\left\vert f^{\prime }\left( \frac{x+b}{2}\right)
\right\vert \right]  \notag
\end{eqnarray}%
for each $x\in \left[ a,b\right] .$
\end{theorem}

\begin{proof}
Suppose that $p>1$. From Lemma \ref{l1} and using the H\"{o}lder's
inequality, we can write%
\begin{eqnarray}
&&\left\vert f\left( x\right) -\frac{1}{b-a}\int_{a}^{b}f\left( u\right)
du\right\vert  \label{26} \\
&\leq &\frac{\left( x-a\right) ^{2}}{b-a}\int_{0}^{1}t\left\vert f^{\prime
}\left( tx+\left( 1-t\right) a\right) \right\vert dt+\frac{\left( b-x\right)
^{2}}{b-a}\int_{0}^{1}t\left\vert f^{\prime }\left( tx+\left( 1-t\right)
b\right) \right\vert dt  \notag \\
&\leq &\frac{\left( x-a\right) ^{2}}{b-a}\left( \int_{0}^{1}t^{p}dt\right) ^{%
\frac{1}{p}}.\left( \int_{0}^{1}\left\vert f^{\prime }\left( tx+\left(
1-t\right) a\right) \right\vert ^{q}dt\right) ^{\frac{1}{q}}  \notag \\
&&+\frac{\left( b-x\right) ^{2}}{b-a}\left( \int_{0}^{1}t^{p}dt\right) ^{%
\frac{1}{p}}.\left( \int_{0}^{1}\left\vert f^{\prime }\left( tx+\left(
1-t\right) b\right) \right\vert ^{q}dt\right) ^{\frac{1}{q}}.  \notag
\end{eqnarray}%
But since $\left\vert f^{\prime }\right\vert ^{q}$ is $h$-concave, using the
inequality (\ref{102}), we have%
\begin{equation}
\left( \int_{0}^{1}\left\vert f^{\prime }\left( tx+\left( 1-t\right)
a\right) \right\vert ^{q}dt\right) \leq \frac{1}{2h\left( \frac{1}{2}\right) 
}\left\vert f^{\prime }\left( \frac{x+a}{2}\right) \right\vert ^{q}
\label{27}
\end{equation}%
and%
\begin{equation}
\left( \int_{0}^{1}\left\vert f^{\prime }\left( tx+\left( 1-t\right)
b\right) \right\vert ^{q}dt\right) \leq \frac{1}{2h\left( \frac{1}{2}\right) 
}\left\vert f^{\prime }\left( \frac{x+b}{2}\right) \right\vert ^{q}.
\label{28}
\end{equation}%
By combining the above numbered inequalities, we obtain%
\begin{eqnarray*}
&&\left\vert f\left( x\right) -\frac{1}{b-a}\int_{a}^{b}f\left( u\right)
du\right\vert \\
&\leq &\frac{\left( x-a\right) ^{2}}{b-a}\frac{1}{\left( p+1\right) ^{\frac{1%
}{p}}}\left( \frac{1}{2h\left( \frac{1}{2}\right) }\right) ^{\frac{1}{q}%
}\left\vert f^{\prime }\left( \frac{x+a}{2}\right) \right\vert \\
&&+\frac{\left( b-x\right) ^{2}}{b-a}\frac{1}{\left( p+1\right) ^{\frac{1}{p}%
}}\left( \frac{1}{2h\left( \frac{1}{2}\right) }\right) ^{\frac{1}{q}%
}\left\vert f^{\prime }\left( \frac{x+b}{2}\right) \right\vert \\
&=&\frac{1}{\sqrt[q]{2}\left( p+1\right) ^{\frac{1}{p}}h^{\frac{1}{q}}\left( 
\frac{1}{2}\right) }\left[ \frac{\left( x-a\right) ^{2}}{b-a}\left\vert
f^{\prime }\left( \frac{x+a}{2}\right) \right\vert +\frac{\left( b-x\right)
^{2}}{b-a}\left\vert f^{\prime }\left( \frac{x+b}{2}\right) \right\vert %
\right]
\end{eqnarray*}%
The proof is completed.
\end{proof}

A midpoint type inequality for functions whose derivatives in absolute value
are $h$-concave may be established from the above result as follows:

\begin{corollary}
In (\ref{25}), if we choose $x=\frac{a+b}{2}$, then we get%
\begin{eqnarray}
&&\left\vert f\left( \frac{a+b}{2}\right) -\frac{1}{b-a}\int_{a}^{b}f\left(
u\right) du\right\vert  \label{29} \\
&\leq &\frac{b-a}{\sqrt[q]{2^{2q+1}}\left( p+1\right) ^{\frac{1}{p}}h^{\frac{%
1}{q}}\left( \frac{1}{2}\right) }\left[ \left\vert f^{\prime }\left( \frac{%
3a+b}{4}\right) \right\vert +\left\vert f^{\prime }\left( \frac{a+3b}{4}%
\right) \right\vert \right] .  \notag
\end{eqnarray}%
For instance, if $h\left( t\right) =t,$ then we obtain%
\begin{eqnarray}
&&\left\vert f\left( \frac{a+b}{2}\right) -\frac{1}{b-a}\int_{a}^{b}f\left(
u\right) du\right\vert  \label{210} \\
&\leq &\frac{b-a}{4\left( p+1\right) ^{\frac{1}{p}}}\left[ \left\vert
f^{\prime }\left( \frac{3a+b}{4}\right) \right\vert +\left\vert f^{\prime
}\left( \frac{a+3b}{4}\right) \right\vert \right] .  \notag
\end{eqnarray}%
where $\left\vert f^{\prime }\right\vert ^{q}$ is $h$-concave function on $%
\left[ a,b\right] ,$ $p,q>1.$ \ \ \ \ \ \ \ \ \ \ \ \ \ \ \ \ \ \ \ \ \ \ \
\ \ \ \ \ \ \ \ \ \ \ \ \ \ \ \ \ \ \ \ \ \ \ \ \ \ \ \ \ \ \ \ \ \ \ \ \ \
\ \ \ \ \ \ \ \ \ \ \ \ \ \ \ \ \ \ \ \ \ \ \ \ \ \ \ \ \ \ \ \ \ \ \ \ \ \
\ \ \ \ \ \ \ \ \ \ \ \ \ \ \ \ 
\end{corollary}

\section{\textbf{Applications to special means}}

We consider the means for arbitrary positive numbers $a,b$ $\left( a\neq
b\right) $ as follows;

The arithmetic mean:%
\begin{equation*}
A\left( a,b\right) =\frac{a+b}{2}
\end{equation*}

The generalized $log$-mean :%
\begin{equation*}
L_{p}\left( a,b\right) =\left[ \frac{b^{p+1}-a^{p+1}}{\left( p+1\right)
\left( b-a\right) }\right] ^{\frac{1}{p}},\text{ }p\in 
%TCIMACRO{\U{211d} }%
%BeginExpansion
\mathbb{R}
%EndExpansion
\smallsetminus \left\{ -1,0\right\} .
\end{equation*}

The identric mean:%
\begin{equation*}
I\left( a,b\right) =\frac{1}{e}\left( \frac{b^{b}}{a^{a}}\right) ^{\frac{1}{%
b-a}}
\end{equation*}

Now, using the result of Section 2, we give some applications to special
means of real numbers.

In \cite{var}, the following example is given:

\begin{example}
\cite{var} Let $h$ be a function defined by $h\left( x\right) =\left(
c+x\right) ^{p-1},$ $x\geq 0.$ If $c=0$, then the function $h$ is
multiplicative. If $c\geq 1$, then for $p\in \left( 0,1\right) $\ the
function\ $h$ is super-multiplicative and for\ $p>1$ the function $h$ is
sub-multiplicative.
\end{example}

Hence, for $c=1,$ $p\in \left( 0,1\right) $, we have $h\left( t\right)
=\left( 1+t\right) ^{p-1},$ $t\geq 0$ is supermultiplicative. Let $f\left(
x\right) =x^{n},$ $x>0,$ $\left\vert n\right\vert \geq 2$ is $h$-convex
functions.

\begin{proposition}
Let $0<a<b$ , $p\in \left( 0,1\right) $ and $\left\vert n\right\vert \geq 2$
. Then%
\begin{equation*}
\left\vert A^{n}\left( a,b\right) -L_{n}^{n}\left( a,b\right) \right\vert
\leq \frac{M\left( b-a\right) }{4}\left[ \int_{0}^{1}\left( 1+t^{2}\right)
^{p-1}dt+\int_{0}^{1}\left( 1+t-t^{2}\right) ^{p-1}dt\right]
\end{equation*}
\end{proposition}

\begin{proof}
The inequality is derived from (\ref{21}) with $x=\frac{a+b}{2}$ applied to
the $h$-convex functions $f:%
%TCIMACRO{\U{211d} }%
%BeginExpansion
\mathbb{R}
%EndExpansion
\rightarrow 
%TCIMACRO{\U{211d} }%
%BeginExpansion
\mathbb{R}
%EndExpansion
,$ $f\left( x\right) =x^{n},$ $\left\vert n\right\vert \geq 2$ \ and $h:%
%TCIMACRO{\U{211d} }%
%BeginExpansion
\mathbb{R}
%EndExpansion
\rightarrow 
%TCIMACRO{\U{211d} }%
%BeginExpansion
\mathbb{R}
%EndExpansion
,$ $h\left( t\right) =\left( 1+t\right) ^{p-1},$ $p\in \left( 0,1\right) .$
The details are disregarded.
\end{proof}

\begin{proposition}
Let $0<a<b$, $p\in \left( 0,1\right) ,$ $q>1$ and $\left\vert n\right\vert
\geq 2$ . Then%
\begin{equation*}
\left\vert A^{n}\left( a,b\right) -L_{n}^{n}\left( a,b\right) \right\vert
\leq \frac{\sqrt[p]{2}M\left( b-a\right) }{8}\left[ \int_{0}^{1}\left(
1+t^{2}\right) ^{p-1}dt+\int_{0}^{1}\left( 1+t-t^{2}\right) ^{p-1}dt\right]
^{\frac{1}{q}}
\end{equation*}
\end{proposition}

\begin{proof}
The inequality is derived from (\ref{23}) with $x=\frac{a+b}{2}$ applied to
the $h$-convex functions $f:%
%TCIMACRO{\U{211d} }%
%BeginExpansion
\mathbb{R}
%EndExpansion
\rightarrow 
%TCIMACRO{\U{211d} }%
%BeginExpansion
\mathbb{R}
%EndExpansion
,$ $f\left( x\right) =x^{n},$ $\left\vert n\right\vert \geq 2$ \ and $h:%
%TCIMACRO{\U{211d} }%
%BeginExpansion
\mathbb{R}
%EndExpansion
\rightarrow 
%TCIMACRO{\U{211d} }%
%BeginExpansion
\mathbb{R}
%EndExpansion
,$ $h\left( t\right) =\left( 1+t\right) ^{p-1},$ $p\in \left( 0,1\right) .$
The details are disregarded.
\end{proof}

\begin{proposition}
Let $0<a<b$ and $p,q>1$. Then we have%
\begin{eqnarray*}
&&\left\vert \ln \left( A\left( a,b\right) +1\right) -\left( b-a\right) \ln
I\left( a+1,b+1\right) \right\vert \\
&\leq &\frac{b-a}{4\left( p+1\right) ^{\frac{1}{p}}}\left[ \frac{1}{3a+b+4}+%
\frac{1}{a+3b+4}\right] .
\end{eqnarray*}
\end{proposition}

\begin{proof}
The inequality is derived from (\ref{210}) applied to the concave function $%
f:\left[ a,b\right] \rightarrow 
%TCIMACRO{\U{211d} }%
%BeginExpansion
\mathbb{R}
%EndExpansion
,$ $f\left( x\right) =\ln \left( x+1\right) .$ The details are disregarded.
\end{proof}

\begin{flushleft}
\textbf{Competing interests}

The author declares that they have no competing interests.
\end{flushleft}

\begin{acknowledgement}
The author gives his warm thanks to the Editor and the Authors for their
precious papers in the reference list.
\end{acknowledgement}

\end{document}